\documentclass[12pt,reqno]{amsart}
\usepackage{amssymb,amsmath, bm,cite,fullpage,mathabx,color,enumerate,multicol}
\usepackage{tikz,tikz-3dplot}
\usepackage[colorlinks=true,linkcolor=blue,citecolor=blue]{hyperref}
\usepackage{bm,fullpage}

\usepackage{tikz}
\usetikzlibrary{shapes,positioning,intersections,quotes}
\usepackage[font=small,labelfont=bf]{caption}

\newtheorem{lemma}{Lemma}
\newtheorem{theorem}{Theorem}
\newtheorem{prop}{Proposition}
\newtheorem{cor}{Corollary}

\theoremstyle{remark}

\newtheorem{remark}{Remark}

\newcommand{\R}{\ensuremath{\mathbb{R}}}

\newcommand{\Z}{\ensuremath{\mathbb{Z}}}

\newcommand{\N}{\ensuremath{\mathbb{N}}}

\title{A framework for discrete bilinear spherical averages and applications to $\ell^p$-improving estimates}
\author[T.C. Anderson]{Theresa C. Anderson}
\address{Department of Mathematics, Carnegie Mellon University, Wean Hall, Hammerschlarg Dr., Pittsburgh, PA 15213}
\email{tanders2@andrew.cmu.edu}
\author[A.V. Kumchev]{Angel V. Kumchev}
\address{Department of Mathematics, Towson University, 8000 York Road, Towson, MD 21252}
\email{akumchev@towson.edu}
\author[E.A. Palsson]{Eyvindur A. Palsson}
\address{Department of Mathematics, Virginia Tech, 225 Stanger Street, Blacksburg, VA 24061}
\email{palsson@vt.edu}

\date{\today}

\numberwithin{equation}{section}

\newcommand{\eps}{\varepsilon}

\begin{document}

\begin{abstract}
We decompose the discrete bilinear spherical averaging operator into simpler operators in several ways.  This leads to a wide array of extensions, such as to the simplex averaging operator, and applications, such as to operator bounds.
\end{abstract}

\maketitle

\section{Introduction}
Pointwise upper bounds are constantly sought after in analysis, as they provide us with direct ways to estimate more complex objects.  It is this theme that we pursue in this paper; inspired by progress in the continuous setting, we seek to develop these tools in the discrete (number-theoretic) setting.  Our main object of study is the discrete bilinear spherical averaging operator, which is not only important in a vairety of applications, but is a natural first object of study in several ways.  Indeed, proving that the bilinear spherical operator is dominated above by a product of two operators (a linear spherical maximal operator and a Hardy-Littlewood maximal function) is only a very recent result \cite{JL}, \cite{AP2}.  The ways balls and spheres interact in number theory is already interesting, and is the first step to considering more complicated Diophantine equations and operators, such as the triangle operator \cite{AKP} which involves more complexity.  Recent breakthroughs in the continuous world have propelled advances in the discrete, where tools and underlying structure are quite different due to the way curvature interacts with lattice point counts.  This already poses significant difficulties in the linear setting in earlier pioneering works such as \cite{MSW}.  Our paper gives both an exact pointwise decomposition and a more refined pointwise upper bound for these bilinear averages, yielding more structural information about the ways linearity, curvature and number theory interact.  Moreover, we can extend this work to simplex operators and apply this decomposition to get a variety of $\ell^p$-improving bounds.

Motivated by the seminal work of Stein \cite{Stein76} and Bourgain \cite{Bourgain86} on the boundedness of the maximal linear spherical averaging operator, Geba {\em et al.} \cite{GebaEtAl13} introduced the maximal bilinear spherical averaging operator. Even earlier, particular variants of it appeared in the work of Oberlin \cite{DOberlin88}. A number of partial results were obtained \cite{BarrionuevoEtAl18,GrafakosEtAl21,HeoEtAl2020} until Jeong and Lee \cite{JL} fully resolved the boundedness of the maximal bilinear spherical averaging operator through a slicing method that allowed for a pointwise domination of the operator by a product of the maximal linear spherial averaging operator and the Hardy-Littlewood maximal operator. There are also recent further refinements obtaining sparse bounds in this setting \cite{PalssonSovine,BFOPZ}.

Just as Magyar, Stein and Wainger established sharp mapping properties for a discrete analogue of the maximal linear spherical averaging operator in their celebrated work \cite{MSW}, the first and third authors of this note, motivated by developments in the continuous setting, introduced a discrete analogue of the maximal bilinear spherical averaging operator \cite{AP,AP2}. In one of those papers, they obtained a sharp range of estimates in high enough dimensions, through a slicing argument analogous to that of Jeong and Lee, while in the other paper, they built on the work of Magyar, Stein and Wainger and many others and employed the circle method to obtain a broad range of estimates in dimensions lower than what the slicing method can handle. The first author has further elaborated on number theoretic variants of the slicing method and the many intricacies that combining number theory, curvature, and multilinearily entail in a recent paper \cite{Anderson}.  One of the contributions of this paper is to extend the reach of the slicing method to handle cases where the circle method was used previously.

Jeong and Lee \cite{JL} not only showed the full boundedness range for the maximal bilinear spherical averaging operator, but also established a wide range of bounds for a single scale variant. In order to establish those bounds a more refined decomposition was needed. Motivated by their work, we refine the pointwise decomposition initiated in the discrete setting in \cite{AP} significantly in two key ways. Both refinements represent a framework in which to view discrete bilinear averages, which we expect to have further applications. Already, several applications appear below. Our techniques also carry over to simplex operators considered in \cite{AKP} and \cite{CLM}.

To state our refinements, first define the discrete bilinear spherical averages by
\[
T_\lambda(f,g)(\bm{x}) = \frac{1}{N(\lambda)}\sum_{|\bm{u}|^2+|\bm{v}|^2=\lambda}f(\bm{x}-\bm{u})g(\bm{x}-\bm{v})
\]
where $\bm{u}$, $\bm{v} \in \Z^d$ and $N(\lambda) = \#\{ (\bm{u},\bm{v}) \in \Z^d\times\Z^d: |\bm{u}|^2+|\bm{v}|^2=\lambda\}$ is the number of lattice points on the sphere of radius $\lambda^{1/2}$ in $\R^{2d}$, asymptotic to $\lambda^{d-1}$ for all $d \geq 3$. Our first refinement decomposes $T_\lambda$ into a weighted average of spherical averages $S_\lambda$ (for a precise definition, see \eqref{spherical average}), which is an equality in higher dimensions.
\begin{prop}
    \label{first refinement}
    For all dimensions $d\geq 3$, we have
    \[
 T_\lambda(f,g) \lesssim \sum_{r=1}^{\lambda - 1} \frac{r^{d/2 - 1}(\lambda - r)^{d/2 - 1}}{\lambda^{d-1}} S_{r}(f)S_{\lambda-r}(g) + \frac{1}{\lambda^{d/2}}fS_{\lambda}(g) + \frac{1}{\lambda^{d/2}}gS_{\lambda}(f).
 \]
 Moreover, if $d \geq 5$, we have equality above.
\end{prop}

Our second refinement, and main result, can now be stated as a decomposition on these bilinear averages as a weighted sum of a ball average $A_\lambda$ and a dyadic maximal spherical average.
\begin{theorem}
\label{second refinement}
    For all dimensions $d\geq 3$, the discrete bilinear spherical averages have the pointwise bound of
        \[
    T_\lambda(f,g) \lesssim \lambda^{1-d/2}A_\lambda (f)g + \lambda^{1-d}\sum_{0 \le k \le \log_2{\lambda}}(2^k)^{d/2-1}(\lambda-2^k)^{d/2}A_{\lambda-2^k}(f){S}_{2^k}^*(g)
    \]
    where ${S}_{2^k}^*$ is the dyadic maximal spherical average defined in \eqref{dyadic spherical}.
\end{theorem}
A more precise definition of the ball averages $A_\lambda(f)$ is given below.  This result is a variant of the decomposition used in the proof of Proposition 3.2 in \cite{JL}, but due to the discrete setting, we utilize new approaches and refinements.  In particular, we significantly refine the approach used in \cite{AP2} to be able to fully capture more subtle upper bounds.

One of the main applications of these decomposition results is to $\ell^p$-improving.  We recall briefly some history of $\ell^p$-improving for spherical averages, though there has been much activity in this area recently for a wide variety of operators. In the continuous setting, $L^p$-improving estimates for the spherical averaging operator are due to Strichartz \cite{Strichartz70} and Littman \cite{Littman71}. Improving estimates for the discrete (linear) spherical averages were first shown in Hughes \cite{Hughes} and Kesler and Lacey \cite{KL}.  Namely, they showed that for $d\geq 5$ and $\frac{d+1}{d-1} < p < 2$, one has
\begin{equation}
\label{linear improving}
   \|S_\lambda f\|_{p'} \lesssim\lambda^{\frac{d}{2} - \frac {d}{p}}\|f\|_{p}, 
\end{equation}
where $\| \cdot \|_p$ denotes the standard norm on $\ell^p(\mathbb Z^d)$, and $p' = p/(p-1)$ is the conjugate exponent to $p$.
(There are inconsistencies in various versions of these references, but we believe the result is correct as stated here.) For more literature on $\ell^p$-improving estimates in the discrete setting see \cite{Kesler21,HLY21,KLM20,Anderson20,MadridNote,HKLMY20}. Combined with the slicing techniques in \cite{AP2}, one can apply the result of Hughes and Kesler-Lacey in the multilinear setting as a black box, yielding multilinear $\ell^p$-improving.  Our results go much further to significantly improve on the black box application of \eqref{linear improving}. 

To state our results, we will frequently switch between language used in $\ell^p$-improving and that used in operator bounds. Recall that an operator $T$ is of type $(p,q;r)$ if $T$ is bounded from $\ell^p\times \ell^q\to \ell^r$, that is,
\[
\|T(f,g)\|_{r} \lesssim \|f\|_{p} \|g\|_{q}.
\]
In our work, the implied constants in such bounds for $T_\lambda$ will depend critically on $\lambda$, and therefore we will always indicate this dependence explicitly (all other implied constants will depend on other fixed parameters).  With this in mind, we note that our first corollary yields some $\ell^p$-improving in the range $p>2$.  This range is impossible in the linear setting.  We state the first corollary in what we think is its most natural form, but a symmetric statement that follows immediately by interpolation is given right after its proof.

\begin{cor}
\label{second refinement corollary}
The bilinear spherical averages satisfy the following $\ell^p$-improving bound
\[
\|T_\lambda(f,g)\|_{1} \lesssim \lambda^{-\frac{d}{2p}}\|f\|_{1}\|g\|_{p}
\]
whenever $d \geq 5$ and $p > \frac{d}{d-2}$.
\end{cor}

One can actually even change the $\ell^1$ norm to an $\ell^r$ norm with $r<1$ via a key lemma, Lemma \ref{lp improving lemma}.  But this lemma will do even more---we improve the range of $\ell^p$-improving that Proposition \ref{first refinement} yields, giving operator bounds for $T_\lambda$ that dip below $\ell^1$.  These are summed up in the next corollary.

\begin{cor}
\label{first refinement corollary}
The bilinear spherical averages satisfy the following $\ell^p$-improving bound
\[
\|T_\lambda(f,g)\|_{s} \lesssim \lambda^{\frac d{2s}-\frac dp}\|f\|_{p}\|g\|_{p}
\]
whenever $d \geq 5$, $\frac{d+1}{d-1}<p<2$, and $\frac p2 \le s \leq 1$.
\end{cor}

One may wonder if we can utilize similar decompositions to get $\ell^p$-improving for operators with more interactions, such as the simplex operators (a precise definition appears later).  We can do this by essentially ignoring the interactions between the variables in the definition in this operator.  Even though the interactions are ignored, the results are still satisfying and give a glimpse at what more refined technology may achieve.  In particular, see Proposition~\ref{prop5} below.  To go beyond this appears quite difficult and will likely involve using a significant amount of number theory to improve the $\ell^p\times \ell^q\to \ell^r$ bounds for the corresponding maximal functions (initiated in the case of triangle averages in \cite{AKP}, where one can already see the significant obstacles).  We are currently planning to pursue this problem in future work.

Multilinear and $k$-spherical analogues of these results are certainly possible. Indeed, we give a glimpse of such possibilities in our discussion of simplex averages in \S\ref{sec:3.4}. In general, such generalizations follow in the same vein as the techniques developed here, and we leave the details to the interested reader.

This paper is organized into five more sections.  Section 2 introduces background and simple, yet important, observations in the discrete setting.  Section 3 includes a basic decomposition.  Though the proof of this only takes one line, there are several applications and extensions that we can derive, including to maximal functions, basic $\ell^p$-improving, lacunary averages, and most notably to simplex operators.  Section 4 has our first refinement as a weighted sum of spherical averages, and discusses the further implications for $\ell^p$-improving.  Section 5 proves our second refinement, Theorem \ref{second refinement}, which is then used to derive Corollary~\ref{second refinement corollary}.  We more fully discuss the implications for operator bounds in Section 6, where the key Lemma \ref{lp improving lemma} is presented and then applied to prove Corollary \ref{first refinement corollary}.  Throughout we include discussion and mention several open directions for investigation.  We underscore that there are many other interesting papers elaborating on the topics presented here in different contexts.

\section{Background and frequently used observations}
Let $B = B(0,\lambda)$ denote the ball in $\Z^d$ with radius $\lambda^{1/2}$, and let $\mu_\lambda \approx \lambda^{-d/2}\bm 1_{B}$ denote the probability measure on this ball.  We will use frequently that 

\[
\|\mu_\lambda\|_q \lesssim \lambda^{-\frac d{2}+\frac{d}{2q}}
\]
with our normalization, via a simple calculation.  We will also frequently use the the nesting properties of the discrete $\ell^p$ spaces: 
\[
\|f\|_{q} \leq \|f\|_{p} \text{ for all } 1 \leq p \leq q \leq \infty.
\]

Recall that $S(\lambda) = \#\{ \bm{x} \in \Z^d: |\bm{x}|^2=\lambda\}$ is the number of lattice points on the sphere of radius $\lambda^{1/2}$ in $\R^{d}$, which is asymptotic to $\lambda^{d/2-1}$ for all $d \geq 5$ by the Hardy-Littlewood asymptotic.  The notation $| \cdot |$ will denote the Euclidean norm on $\mathbb R^d$.  For dimensions $3$ and $4$ we cannot use this asymptotic, but we will be able to get around this in certain cases.  We can then define a probability measure on the sphere of radius $\lambda$ in $\Z^d$ as $\sigma_\lambda = S(\lambda)^{-1}\bm 1_{\{|\bm x|^2 = \lambda\}}$.  To avoid having to write absolute values, we will assume that all functions are nonnegative, however, all our results carry over the general functions by tracking these absolute values.
We can then define the (discrete) ball averages as $A_\lambda(f) = f*\mu_\lambda$ and the discrete spherical averages as 
\begin{equation}
\label{spherical average}
    S_\lambda(f) = f*\sigma_\lambda = \frac{1}{S(\lambda)}\sum_{|\bm{u}|^2=\lambda}f(\bm{x}-\bm{u}).
\end{equation}

While this paper is about averaging operators, we will also use the spherical maximal function
\[
S^*(f) = \sup_{\lambda\in \N}S_\lambda(f).
\]
Finally, we use the convention to list $\ell^p$-improving estimates for the range in which there is decay in terms of $\lambda$ (i.e., the radius) in our statements, even if the estimate holds (without decay) outside of that range.

\section{First decomposition and implications}

As a first result, we can slightly refine the pointwise decomposition from \cite{AP2}, extending the range of bounds there.  We have 
\[
|T_\lambda (f,g)(\bm{x})| \lesssim \frac{1}{\lambda^{\frac{d}{2}}}\sum_{|\bm{u}|^2 \leq \lambda}|f(\bm{x}-\bm{u})|\cdot \frac{1}{\lambda^{\frac{d}{2}-1}} \bigg|\sum_{|\bm{v}|^2 = \lambda - |\bm{u}|^2} g(\bm{x}-\bm{v})\bigg| \leq A_\lambda(f)(\bm{x})\cdot S^* (g)(\bm{x})
\]
Already this simple bound has several applications and extensions, which we briefly consider before moving on to a more delicate decomposition.
\subsection{Bilinear maximal function in dimensions $3$ and $4$}
First, we note an extension of the results in \cite{AP2} for the discrete bilinear spherical maximal function $T^*(f,g) = \sup_{\lambda\in\N} T_\lambda(f,g)$ in dimensions 3 and 4. Using the pointwise domination we immediately get
\begin{equation}
 \|T^*(f,g)\|_{1+\varepsilon} \lesssim \|A^*(f)\|_{1+\varepsilon}\|S^* (g)\|_{\infty}   
\end{equation}
where $A^*$ is the Hardy-Littlewood maximal function.

Using the fact that the spherical maximal function is bounded on $\ell^\infty$, even in dimensions $3$ and $4$, and the Hardy-Littlewood maximal function on $\ell^{1+\varepsilon}$, we have
$T^*$ is bounded of type $(1+\varepsilon, \infty; 1+\varepsilon)$.  By switching the roles of $f$ and $g$, we have also that $T^*$ is bounded of type $(\infty, {1+\varepsilon}; {1+\varepsilon})$.  Interpolation then yields bounds in dimensions 3 and 4, not addressed in \cite{AP2}, and significantly improves the bounds in \cite{AP}.

\begin{prop}
The operator $T^*$ is bounded of type $(p,q;r)$ for all $p,q,r>1$ with $\frac{1}{p}+\frac{1}{q} \geq \frac{1}{r}$.
\end{prop}

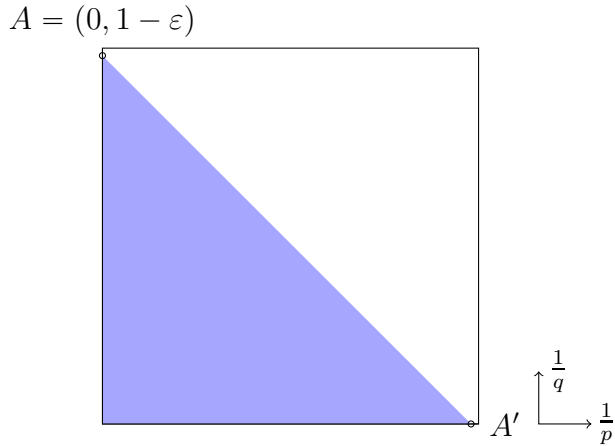
\begin{figure}[h!]\label{range}
\begin{tikzpicture}
\draw (0,0) rectangle (5,5);
\path[fill=blue!35] (0,0)--(0,4.9)--(4.9,0);
\draw (0,5)--(0,0)--(5,0);
\node [above] at (0,5) {$A=(0,1-\varepsilon)$};
\draw (0,4.9) circle [radius=0.04];
\node [right] at (5,0) {$A'$};
\draw (4.9,0) circle [radius=0.04];
\draw [<->] (5.8,0.7)--(5.8,0)--(6.5,0);
\node at (6.7,0) {$\frac{1}{p}$};
\node at (5.8,0.7) [right]{$\frac{1}{q}$};
\end{tikzpicture}
\caption{The range of $\ell^p\times \ell^q$ bounds (in terms of $\frac{1}{p}$ and $\frac{1}{q}$) for the discrete bilinear spherical maximal function for $d=3,4$.  The notation $X'$ represents the symmetric point to $X$ about the diagonal.}
\end{figure}
\subsection{Bilinear averages}
Secondly, by the pointwise product decomposition above and H\"older's inequality, we have
\[
\|T_\lambda(f,g)\|_{1} \lesssim \|A_\lambda (f)\|_{{p'}}\|S^* (g)\|_{p}
\]
$S^*$ is bounded on $\ell^p$ for all $p>\frac{d}{d-2}$ and $d \geq 5$, and the ball averages satisfy $\ell^p$-improving estimates via Young's inequality:
\begin{equation}\label{eq3.2}
 \|f*\mu_\lambda\|_{{p'}} \lesssim \|f\|_{p}\|\mu_\lambda\|_{{p'/2}} = \lambda^{\frac{d}{2}-\frac{d}{p}}\|f\|_{p}.   
\end{equation}
Hence, we get the following proposition.

\begin{prop}
When $d \ge 5$, the operator $T_\lambda$ is bounded of type $(p,p;1)$ for all $\frac{d}{d-2}<p< 2$, and it satisfies the following improving estimate:
\begin{equation}
\|T_\lambda(f,g)\|_{1} \lesssim \lambda^{\frac{d}{2}-\frac{d}{p}}\|f\|_{p}\|g\|_{p}. \label{second decomposition improving}
\end{equation}
\end{prop}
This bound has decay in $\lambda$ when $\frac{d}{d-2} < p < 2$.  A main point of our paper is that we will be able to improve upon this significantly later on.

\subsection{Lacunary bilinear maximal function}

There have been several recent results on lacunary maximal functions in the continuous and discrete setting.  Lacunary maximal functions are interesting to study in order to understand how far one can push a full maximal bound, and when the full maximal bounds are known, lacunary ones are often not.  In particular, a wider range of $\ell^p$ bounds can be achieved when restricting to lacunary sequences, and hence the best current results for the (linear) discrete lacunary spherical maximal function improve over the range for the full one.  It is an open question to find sharp bounds for the discrete lacunary spherical maximal function; the current best bounds are in \cite{KLM}, \cite{CH}.  Along these lines we list the simple bound that follows from combining the current best linear bounds with our first basic decomposition.  

Let $\{\lambda_k\}$ be a lacunary sequence.  Observe that 
\[
\|\sup_{\lambda_k}T(f,g)\|_{1} \leq \sum_{\lambda_k} \|T_{\lambda_k}(f,g)\|_{1} \leq \sum_{\lambda_k} \lambda_k^{-\frac{d}{2}+\frac{d}{p'}}\|f\|_{p}\|g\|_{p}.
\]
So, if $d\geq 6$ and $\frac{d-2}{d-3}<p<2$, the sum in $\lambda$ is finite and the bilinear discrete lacunary maximal function is of type $(p,p;1)$.  While the lower bound on $p$ is lower than those appearing elsewhere in this paper (for $d \geq 6$), this is only because of a better known linear range.  Lacunary bounds are not the main topic of this paper, but perhaps by using some of our ideas, one could obtain better bilinear discrete lacunary maximal bounds, improving significantly on the above.

\subsection{Simplex operators}
\label{sec:3.4}

We can also apply the same technique to the (discrete) simplex averaging operator, considered in \cite{AKP} and \cite{CLM} and recalled below in \eqref{simplex operator}.  As mentioned in the introduction, one of the main intricacies about this operator is that the variables $\bm{u}$ and $\bm{v}$ interact.  Due to the subtleties of this interaction, bounding this operator in the triangle case ($k=2$) and maintaining these interactions was a highly nontrivial task, which we undertook in \cite{AKP}.  Cook, Lyall, and Magyar approach the problem differently. They bypass the interactions, which allows them to treat all $k$ in both the continuous and discrete setting, though this results in a smaller range of bounds.  We can essentially use this approach here to yield $\ell^p$-improving bounds for these operators, which we state, prove, and discuss below. 

Let
\begin{equation}
T_{\rm tri}(f,g) = (\#\mathcal V_\lambda)^{-1} \sum_{(\bm{u,v}) \in \mathcal V_\lambda} f(\bm x-\bm u)g(\bm x-\bm v),
\end{equation}
where the summation is over the point set
\begin{align*}
\mathcal V_\lambda &= \big\{ (\bm{u, v}) \in \mathbb Z^d \times \mathbb Z^d : |\bm u|^2 = |\bm v|^2 = |\mathbf u - \mathbf v|^2 = \lambda \big\} \\
&= \big\{ (\bm{u, v}) \in \mathbb Z^d \times \mathbb Z^d : |\bm u|^2 = |\bm v|^2 = 2\bm u\cdot \bm v = \lambda \big\}.
\end{align*} 
By work of \cite{Raghavan}, \cite{EV}, \cite{AKP}, among many others, we have, for all $d \geq 7$, 
\begin{equation}\label{triangle opearator}
T_{\rm tri}(f,g)(\bm x) \lesssim \frac{1}{\lambda^{d-3}} \sum_{(\bm{u,v}) \in \mathcal V_\lambda} f(\bm x-\bm u)g(\bm x-\bm v).
\end{equation}
This operator is a discrete average over (equilateral) triangles in $\Z^d$, thus bounds for this operator and its maximal variant give significant quantitative information about the number of such triangles in Euclidean space.  

\begin{prop}\label{prop3}
Let $d \ge 7$. The triangle operator satisfies the following $\ell^p$-improving estimate
\[
\|T_{\rm tri}(f,g)\|_{1} \lesssim
\lambda^{1+\frac{d}{2}-\frac{d}{p}}\|f\|_{p}\|g\|_{p}
\]
for $\frac{d+1}{d-1}<p<\frac{2d}{d+2}$.
\end{prop}

\begin{proof}
    Note that
\begin{align*}
T_{\rm tri}(f,g)(\bm{x}) &\lesssim {\lambda^{3-d}}\sum_{|\bm{u}|^2 = \lambda}\sum_{|\bm{v}|^2= |\bm{v-u}|^2 = \lambda}f(\bm{x}-\bm{u})g(\bm{x}-\bm{v})
\\
&\lesssim \lambda \bigg( \frac{1}{\lambda^{d/2-1}}\sum_{|\bm{u}|^2 = \lambda} f(\bm{x}-\bm{u}) \bigg) \bigg( \frac{1}{\lambda^{d/2-1}}\sum_{ |\bm{v}|^2 =\lambda} g(\bm{x}-\bm{v})  \bigg) \\
&\lesssim \lambda S_\lambda(f)(\bm x) S_\lambda (g)(\bm x).    
\end{align*}
Combining this inequality, H\"older's inequality, and \eqref{linear improving} gives
\[
\|T_{\rm tri}(f,g)\|_{1} \lesssim \lambda \|S_\lambda(f)\|_{p'}\|S_\lambda(g)\|_{p} \lesssim \lambda^{1+\frac{d}{2}-\frac{d}{p}}\|f\|_{p} \|S_\lambda(g)\|_{p}
\]
for $\frac{d+1}{d-1} < p < 2$ and $d \ge 5$. The result then follows, since the averaging operator $S_\lambda$ is bounded on $\ell^p$ for $p > 1$.
\end{proof}

For example, at the point $p = \frac{d}{d-2}$, $p' = \frac{d}{2}$, the exponent on $\lambda$ is negative when $d \geq 7$, yielding $l^p$-improving results in this range.  

We can also analyze simplices by defining an analogous operator: for a precise derivation of this representation see \cite{CLM}. Let 
\begin{equation}\label{simplex operator}
T_{\rm sim}(f_1, \dots , f_k)(\bm x) =
\frac{1}{\lambda^{\frac{dk-k(k+1)}{2}}}\sum_{(\bm{u}_1, \dots, \bm{u}_k)\in \mathcal W_\lambda}f_i(\bm x-\bm u_i),
\end{equation}
where 
\begin{equation*}
\mathcal W_\lambda = \big\{ (\bm u_1, \dots, \bm u_k) \in \mathbb Z^d \times \cdots \times \mathbb Z^d = \Z^{kd}: |\bm u_i|^2 = |\bm u_i - \bm u_j|^2 = \lambda \textrm{ for all } i \neq j\big\}.
\end{equation*} 
In this operator the variables interact in $\frac{k(k-1)}{2}$ ways; the fact that we ignore these interactions means that the $\ell^p$-improving estimates worsen as $k$ increases, resulting in a more restricted range on $d$. We summarize the results below.

\begin{prop}\label{prop4}
Let $k \ge 3$ and $d \ge 2k+3$. The simplex operator satisfies the following pointwise bound:
\begin{equation}\label{eq3.5} 
T_{\rm sim}(f_1, \dots , f_k)(\bm x) \lesssim \lambda^{\frac{k(k-1)}{2}}\prod_{i=1}^kS_\lambda(f_i)(\bm x).
\end{equation}

If $2 < p_1 \le p_2 \le \dots \le p_k$ are such that $\sum_i p_i^{-1} = 1$, one has the $\ell^p$-improving bound
\[
\| T_{\rm sim}(f_1, \dots , f_k) \|_1
\lesssim \lambda^{\frac{k(k-1) - dj}{2} + \sum_{i \le j} \frac d{p_i}} \bigg( \prod_{i=1}^j\|f_i\|_{p_i'} \bigg) \bigg( \prod_{i=j+1}^k \|f_i\|_{p_i} \bigg),
\]
where $j$ is the largest index $i$ for which $p_i < \frac 12(d+1)$.  
\end{prop}

\begin{proof}
The first claim is immediate:
\[
  T_{\rm sim}(f_1, \dots , f_k)(\bm x) \le \lambda^{\frac{k(k+1)}{2}-\frac{dk}{2}} \prod_{i=1}^k\sum_{|\bm u_i|^2 = \lambda}f_i(\bm x-\bm u_i)
\lesssim \lambda^{\frac{k(k-1)}{2}}\prod_{i=1}^kS_\lambda(f_i)(\bm x).
\]
By taking the $\ell^1$-norm, followed by H\"older's inequality, we get 
\begin{align*}
\|T_{\rm sim}(f_1, \dots, f_k)\|_1 &\lesssim \lambda^{\frac{k(k-1)}{2}}\|S_\lambda(f_1) \cdots S_\lambda(f_k) \|_1 \lesssim \lambda^{\frac{k(k-1)}{2}}\prod_{i=1}^k\|S_\lambda(f_i)\|_{p_i},
\end{align*}
and the second claim follows by applying \eqref{linear improving} with $p = p_i'$ to the indices $i \le j$. We note that these are exactly the indices with $\frac {d+1}{d-1} < p_i' < 2$.
\end{proof}

\begin{remark}
Note that in the second part of the above proposition, such an index $j \ge 1$ must exist, since the inequality $p_1 \ge \frac 12(d+1)$ is inconsistent with the other hypotheses. Indeed, when $d \ge 4k-4$, one must have $j \ge 2$, since otherwise
\[ 1 = \sum_{i=1}^k \frac 1{p_i} < \frac 12 + \frac {2(k-1)}{d+1} < \frac 12 + \frac 12 = 1. \]
Thus, the exponent of $\lambda$ is no larger than
\[ \frac {k(k-1)}2 - d + \frac d{p_1} + \frac d{p_2} = \frac {k(k-1)}2 - d \sum_{i=3}^k \frac 1{p_i}, \]
which is negative for $d \gtrsim_{p_1,p_2} k^2$. We should point out that if $2 < p_1 < p_2 < 2 + \delta$, the implied constant grows as $\delta^{-1}$, so it may in some cases dominate the restriction on the dimension. 

On the other hand, if $p_1^{-1} + p_2^{-1} < 1-c$ and $d \ge 2c^{-1}(k-2)$, for some absolute constant $c > 0$, a variant of the above comments yields that $j \ge 3$. Hence, the exponent of $\lambda$ is then no larger than
\[ \frac {k(k-1)}2 - \frac d2 - d \sum_{i=4}^k \frac 1{p_i}, \]
and this expression is negative for $d > k(k-1)$. In the same vein, it is instructive to consider what happens in the ``balanced regime'' when $p_k < \frac 12(d+1)$ (so, $j=k$). This includes, for example, the case $p_1 = \dots = p_k = k$. Then the exponent of $\lambda$ is 
\[ \frac {k(k-1)}2 - \frac {dk}2 +  \sum_{i=1}^k \frac d{p_i} = \frac {k(k-1) - d(k-2)}2, \]
which is negative for all $d \ge 2k+3$. Thus, in this case, our result is improving in all allowed dimensions.
\end{remark}


Upon a brief reflection, one sees that Proposition \ref{prop4} is not really the case $k \ge 3$ of Proposition \ref{prop3}. Indeed, a result like Proposition \ref{prop4} is not possible for $k=2$. Our next proposition is the rightful extension of Proposition \ref{prop3} to larger values of $k$.  Note that one, and only one, of the indices is less than $2$ in what follows.

\begin{prop}\label{prop5}
Let $k \ge 3$ and $d \ge 2k+3$. If $\frac {d+1}{d-1} < p_1 < 2 < p_2 \le \dots \le p_k$ are such that $\sum_i p_i^{-1} = 1$, one has the $\ell^p$-improving bound
\[
\| T_{\rm sim}(f_1, \dots , f_k) \|_1
\lesssim \lambda^{\frac{k(k-1) - dj}{2} + \frac d{p_1'} + \sum_{1 < i \le j} \frac d{p_i}}  \|f_1\|_{p_1} \prod_{i=2}^k\|f_i\|_{p_i'},
\]
where $j$ is the largest index $i \le k-1$ for which $p_i < \frac 12(d+1)$. Moreover, if $j = k-1$ and $\frac {d+3}{d+1} < 2p_1^{-1} + p_k^{-1} < \frac 32$, one has 
\[
\| T_{\rm sim}(f_1, \dots , f_k) \|_1 \lesssim \lambda^{\frac{k(k-1) - d(k-2)}{2}}  \|f_1\|_{p_1} \prod_{i=2}^k\|f_i\|_{p_i'}.
\]
\end{prop}

\begin{proof}
Similarly to the proof of Proposition \ref{prop4}, we can combine \eqref{eq3.5}, H\"older's inequality, and \eqref{linear improving} with $p = p_1$ to obtain
\begin{align*}
\|T_{\rm sim}(f_1, \dots, f_k)\|_1 &\lesssim \lambda^{\frac {k(k-1)}2}\|S_\lambda(f_1) \cdots S_\lambda(f_k) \|_1 \\
&\lesssim \lambda^{\frac {k(k-1)}2}\|S_\lambda(f_1)\|_{p_1'} \|S_\lambda(f_2) \cdots S_\lambda(f_{k}) \|_{p_1} \\
&\lesssim \lambda^{\frac {k(k-1)}2 - \gamma_1} \|f_1\|_{p_1} \big\| S_\lambda(f_2)^{p_1} \cdots S_\lambda(f_{k})^{p_1} \big\|_1^{1/p_1},
\end{align*}
where $\gamma_1 = \frac d2 - \frac d{p_1'}$. 

Next, we set $q_i = p_i/p_1$, $i = 2, \dots, k-1$, and we note that
\[ 0 < \sum_{i=2}^{k-1} \frac 1{q_i} = p_1 \sum_{i=2}^{k-1}\frac 1{p_i} = p_1 \left( 1 - \frac 1{p_1} - \frac 1{p_k} \right) = p_1 - 1 - \frac {p_1}{p_k} < 1. \]
Thus, it is possible to choose a complementary $q_k \in (1, \infty)$ so that $\sum_i q_i^{-1} = 1$. In particular, $q_k^{-1} = 2 - p_1 + p_1/p_k$, and so we have
\[ \frac 1{p_1q_k} = \frac 2{p_1} - 1 + \frac 1{p_k} < 1 - \frac 1{p_k} = \frac 1{p_k'}. \]
H\"older's inequality then yields
\[ \big\| S_\lambda(f_2)^{p_1} \cdots S_\lambda(f_{k})^{p_1} \big\|_1 \le \prod_{i=2}^k \big\| S_\lambda(f_i)^{p_1}  \big\|_{q_i} = \prod_{i=2}^k \| S_\lambda(f_i)  \|_{p_1q_i}^{p_1} . \]
Thus, when $j \le k-1$, we can use \eqref{linear improving} with $p=p_i'$, $1 < i \le j$, and the boundedness of $S_\lambda$ to deduce that 
\begin{align}
\|T_{\rm sim}(f_1, \dots, f_k)\|_1 &\lesssim \lambda^{\frac {k(k-1)}2 - \gamma_1} \|f_1\|_{p_1} \| S_\lambda(f_k)\|_{p_1q_k} \prod_{i=2}^{k-1} \| S_\lambda(f_i)  \|_{p_i} \notag \\
&\lesssim \lambda^{\frac {k(k-1)}2 - \gamma_2}  \|f_1\|_{p_1} \| f_k \|_{p_1q_k} \bigg( \prod_{i=2}^j \|f_i\|_{p_i'} \bigg) \bigg( \prod_{i=j+1}^{k-1} \| f_i \|_{p_i} \bigg) \notag \\
&\lesssim \lambda^{\frac {k(k-1)}2 - \gamma_2}  \|f_1\|_{p_1} \prod_{i=2}^k \|f_i\|_{p_i'}, \label{eq3.6}
\end{align}
where $\gamma_2 = \gamma_1 + \sum_{1 < i \le j} \big( \frac d2 - \frac d{p_i} \big)$ (and the last step relies on the observation that $q_kp_1 > p_k'$). 
This completes the proof of the first claim of the proposition. 

Finally, when $j = k-1$, the condition $\frac {d+3}{d+1} < 2p_1^{-1} + p_k^{-1} < \frac 32$, means that $2 < p_1q_k < \frac 12(d+1)$, so we may also apply \eqref{linear improving} with $p = (p_1q_k)' > p_k'$ to show that
\[ \| S_\lambda(f_k)\|_{p_1q_k} \lesssim \lambda^{-\frac d2 + \frac d{p_1q_k}} \| f_k \|_{(p_1q_k)'} \le \lambda^{-\frac d2 + \frac d{p_1q_k}} \| f_k \|_{p_k'}. \]
This yields a variant of \eqref{eq3.6} with $\gamma_2$ replaced by
\[ \gamma_3 = \frac {dk}2 - \frac d{p_1'} - \sum_{i=2}^{k-1} \frac d{p_i} - \frac d{p_1q_k} = \frac {dk}2 - \frac d{p_1'} - \sum_{i=2}^{k} \frac d{p_1q_i} = \frac {dk}2 - d. \]
This establishes the second claim of the proposition.  
\end{proof}

\section{First refinement}
\label{sec:4}

Recall that $S(\lambda) \approx \lambda^{d/2 - 1}$ for $d\geq 5$ and $N(\lambda) \approx \lambda^{d-1}$ for $d \ge 3$. Without loss of generality for $d \geq 5$ we can thus consider the discrete spherical averaging operator to be
\begin{equation}
\label{spherical average high dim}
    S_\lambda(f)(\bm{x}) = \frac{1}{\lambda^{d/2 -1}}\sum_{|\bm{u}|^2=\lambda}f(\bm{x}-\bm{u}).
\end{equation}
and for $d\geq 3$ the discrete bilinear spherical averaging operator is
\[
T_\lambda(f,g)(\bm{x}) = \frac{1}{\lambda^{d-1}}\sum_{|\bm{u}|^2+|\bm{v}|^2=\lambda}f(\bm{x}-\bm{u})g(\bm{x}-\bm{v}).
\]

\begin{proof}[Proof of Proposition \ref{first refinement}]
    Our first refinement breaks the discrete bilinear spherical averaging operator up into weighted sums of products of discrete spherical averages.  We begin by assuming that $d \geq 5$ and thus we can use \eqref{spherical average high dim}.  We comment on the changes for $d \geq 3$ at the end.  We have
\begin{align*}
T_\lambda(f,g)(\bm{x})&= \frac{1}{\lambda^{d-1}} \sum_{r=1}^{\lambda - 1} \bigg( \sum_{|\bm{u}|^2 = r}f(\bm{x}-\bm{u}) \bigg) \bigg( \sum_{|\bm{v}|^2 = \lambda - r}g(\bm{x}-\bm{v}) \bigg) \\
& \quad + \frac{1}{\lambda^{d-1}}f(\bm{x})\sum_{|\bm{v}|^2 = \lambda}g(\bm{x}-\bm{v}) + \frac{1}{\lambda^{d-1}}g(\bm{x})\sum_{|\bm{u}|^2 = \lambda}f(\bm{x}-\bm{u})\\
&= \sum_{r=1}^{\lambda - 1} \frac{r^{d/2 - 1}(\lambda - r)^{d/2 - 1}}{\lambda^{d-1}} S_{r}(f)(\bm{x})S_{\lambda-r}(g)(\bm{x}) \\
& \quad + \frac{1}{\lambda^{d/2}}f(\bm x)S_{\lambda}(g)(\bm{x}) + \frac{1}{\lambda^{d/2}}g(\bm x)S_{\lambda}(f)(\bm{x}).
\end{align*}

Note that in dimensions $3$ and $4$ we must use the definition of $S_\lambda$ given in \eqref{spherical average}, since the distribution of points on these spheres is irregular.  However, we can follow the proof above, replacing \eqref{spherical average high dim} with \eqref{spherical average} instead, since \eqref{spherical average} is an upper bound for \eqref{spherical average high dim}.  Thus the decomposition carries through in these dimensions with $=$ replaced by $\lesssim$.  
\end{proof}

This decomposition now allows us to obtain our second set of $\ell^p$ improving bounds that are even better than in the previous section. 

Combining Proposition \ref{first refinement}, H\"{o}lder's inequality, and \eqref{linear improving}, we now obtain that
\begin{align*}
\| T_\lambda(f,g) \|_{1} &\leq \sum_{r=1}^{\lambda - 1} \frac{r^{d/2 - 1}(\lambda - r)^{d/2 - 1}}{\lambda^{d-1}} \| S_r(f) \|_{{p'}} \| S_{\lambda - r}(g) \|_{p} \notag\\
&\qquad + \frac{1}{\lambda^{d/2}}\| f \|_{p} \| S_{\lambda}(g) \|_{{p'}} + \frac{1}{\lambda^{d/2}}\| g \|_{p} \| S_{\lambda}(f) \|_{{p'}} \notag\\
&\lesssim \left( \sum_{r=1}^{\lambda - 1} \frac{r^{d/{2}-1}}{\lambda^{d/2}}r^{ - \frac{d}{p} + \frac d{2}} + \frac{1}{\lambda^{d/2}}\cdot\lambda^{ - \frac{d}{p} + \frac d{2}} \right) \| f \|_{p} \| g \|_{p} \notag\\
&\lesssim \left( \frac{\lambda^{d/{p'}}}{\lambda^{d/2}} + \lambda^{ - \frac{d}{p}} \right) \| f \|_{p} \| g \|_{p} \lesssim \lambda^{\frac d2 - \frac d{p}} \| f \|_{p} \| g \|_{p} 
\end{align*}
for $\frac{d+1}{d-1} < p < 2$. This is a slight improvement over our earlier result, as now the range of $\frac{d+1}{d-1} < p < 2$ is slightly larger than before when $\frac{d}{d-2} < p < 2$. \\

We can also use Proposition \ref{first refinement} to obtain a (slightly restricted) result in dimension $d = 4$. Recall \eqref{linear improving}. Hughes \cite{Hughes} proved also that when $\lambda$ is odd, $d = 4$, and $\frac 53 < p < 2$, one has
\begin{equation}
\label{linear improving+}
   \|S_\lambda f\|_{p'} \lesssim_\eps \lambda^{2 - \frac 4{p} + \eps}\|f\|_{p}. 
\end{equation}
Thus, for $d = 4$ and an odd $\lambda$, we can modify the above argument to get
\begin{align*}
\| T_\lambda(f,g) \|_{1} &\lesssim \sum_{\substack{r=1\\r \text{ odd}}}^{\lambda - 1} \frac{r(\lambda - r)}{\lambda^3} \| S_r(f) \|_{{p'}} \| S_{\lambda - r}(g) \|_{p}  + \sum_{\substack{r=1\\r \text{ even}}}^{\lambda - 1} \frac{r(\lambda - r)}{\lambda^3} \| S_r(f) \|_{{p}} \| S_{\lambda - r}(g) \|_{p'} \\
&\qquad + \frac{1}{\lambda^2}\| f \|_{p} \| S_{\lambda}(g) \|_{{p'}} + \frac{1}{\lambda^2}\| g \|_{p} \| S_{\lambda}(f) \|_{{p'}} \\
&\lesssim_\eps \frac 1{\lambda^3} \left( \sum_{r=1}^{\lambda - 1} r^{3 - \frac 4{p} + \eps}(\lambda-r) + \sum_{r=1}^{\lambda - 1} r(\lambda-r)^{3 - \frac 4{p} + \eps} + \lambda \cdot \lambda^{2 - \frac{4}{p} + \eps} \right) \| f \|_{p} \| g \|_{p} \\
&\lesssim_\eps \frac 1{\lambda^2} \left( \sum_{r=1}^{\lambda - 1} r^{3 - \frac 4{p} + \eps} + \lambda^{2 - \frac{4}{p} + \eps} \right) \| f \|_{p} \| g \|_{p} \lesssim_\eps \lambda^{2 - \frac 4{p} + \eps} \| f \|_{p} \| g \|_{p}.
\end{align*}
Note that in this version, we apply H\"older's inequality differently to the even and odd values of $r$. In each case, this leaves us to apply Hughes' bound \eqref{linear improving+} to a spherical average---$ S_{\lambda-r}$ and $S_r$, respectively---with an odd radial parameter. This ensures that \eqref{linear improving+} is always applicable.\\

\begin{prop}\label{prop7}
When $d \ge 5$, the operator $T_\lambda$ is bounded of type $(p,p;1)$ for all $\frac{d+1}{d-1}<p< 2$, and it satisfies the following improving estimate:
\begin{equation}
\|T_\lambda(f,g)\|_{1} \lesssim \lambda^{\frac{d}{2}-\frac{d}{p}}\|f\|_{p}\|g\|_{p}. \label{second decomposition improving}
\end{equation}
Moreover, when $d=4$, $\frac 53 < p < 2$, and $\lambda$ is odd, $T_\lambda$ satisfies 
\begin{equation}
\|T_\lambda(f,g)\|_{1} \lesssim_\eps \lambda^{2-\frac{4}{p} + \eps}\|f\|_{p}\|g\|_{p}, \label{second decomposition improving-a}
\end{equation}
for any fixed $\eps > 0$.
\end{prop}

\section{Second refinement and implications}

Here we prove our second refinement, Theorem \ref{second refinement}.

\begin{proof}[Proof of Theorem \ref{second refinement}]
Assume for now that $d \geq 5$, and choose $K$ to be the unique integer such that $2^{K} < \lambda \le 2^{K+1}$. We have
\begin{align*}
T_\lambda(f,g)(\bm x) &= \frac{1}{\lambda^{d-1}}\sum_{|\bm u|^2\leq\lambda}\sum_{|\bm v|^2 = \lambda-|\bm{u}|^2}f(\bm x-\bm u)g(\bm x-\bm v)\\
&\le \frac {g(\bm x)}{\lambda^{d-1}}\sum_{|\bm u|^2 = \lambda}f(\bm x-\bm u) + \frac{1}{\lambda^{d-1}}\sum_{k=0}^{K}\sum_{\lambda-|\bm{u}|^2\approx2^k}f(\bm x-\bm u)\sum_{|\bm v|^2 = \lambda-|\bm u|^2}g(\bm x-\bm v)\\
 &  \lesssim\lambda^{1-d/2}g(\bm x)A_\lambda(f) (\bm x) + \frac{1}{\lambda^{d-1}}\sum_{k=0}^{K}\sum_{\lambda-|\bm{u}|^2\approx2^k}f(\bm x-\bm u)\sum_{|\bm v|^2 = \lambda-|\bm u|^2}g(\bm x-\bm v),
\end{align*}
where $\lambda - |\bm u|^2 \approx 2^k$ means $2^k \leq \lambda - |\bm u|^2 < 2^{k+1}$. Note that we could have written the first term instead as $\lambda^{-d/2}g(\bm x)S_\lambda (\bm x)$ and maintained equality (for $d \geq 5$).  However, we will shortly be using an upper bound everywhere, and this form is more suited to our applications. 

We now focus our attention on the last sum (over $k, \bm u, \bm v$). Define this dyadic maximal operator
\begin{equation}
\label{dyadic spherical}
    {S}_{\Lambda}^*(g)(\bm x) =  \sup_{\lambda \approx \Lambda}{\lambda^{1-d/2}}\sum_{|\bm v|^2 = \lambda}g(\bm x-\bm v).
\end{equation}
The sum in question equals
\begin{align*}
&\frac{1}{\lambda^{d-1}}\sum_{k=0}^{K}\sum_{\lambda-|\bm{u}|^2\approx2^k}f(\bm x-\bm u) (\lambda-|\bm u|^2)^{d/2-1} S_{\lambda-|\bm u|^2}(g)(\bm x)\\
&\leq \frac{1}{\lambda^{d-1}}\sum_{k=0}^{K}\sum_{\lambda-|\bm{u}|^2\approx2^k}f(\bm x-\bm u) (\lambda-|\bm u|^2)^{d/2-1} S_{2^k}^*(g)(\bm x)\\
&\lesssim \frac{1}{\lambda^{d-1}}\sum_{k=0}^{K} (2^k)^{d/2-1}  S_{2^k}^*(g)(\bm x) \sum_{\lambda-|\bm{u}|^2\approx2^k}f(\bm x-\bm u).
\end{align*}
Hence,
\begin{align*}
T_\lambda(f,g)(\bm x) \lesssim\lambda^{1-d/2}g(\bm x)A_\lambda(f) (\bm x) + \lambda^{1-d}\sum_{k=0}^{K} (2^k)^{d/2-1}  S_{2^k}^*(g)(\bm x) \sum_{\lambda-|\bm{u}|^2\approx2^k}f(\bm x-\bm u).
\end{align*}
When $\lambda-|\bm u|^2\approx2^k$, we have $|\bm u|^2 \le \lambda-2^k$, so 
\[
\sum_{\lambda-|\bm{u}|^2\approx2^k}f(\bm x-\bm u) \lesssim (\lambda-2^k)^{d/2}A_{\lambda-2^k}(f)(\bm x).
\]
Altogether, we find that
\[
    T_\lambda(f,g)(\bm x) \lesssim\lambda^{1-d/2}g(\bm x)A_\lambda(f) (\bm x) + \lambda^{1-d}\sum_{k=0}^{K} (2^k)^{d/2-1}(\lambda-2^k)^{d/2} A_{\lambda-2^k}(f)(\bm x) S_{2^k}^*(g)(\bm x).
\]

Again, if $d=3,4$ we can replace \eqref{spherical average high dim} with \eqref{spherical average} to get an upper bound, as in the proof of Theorem \ref{first refinement}.
\end{proof}

To take advantage of the above decomposition, we will rely on the $\ell^p$-improving bound for the dyadic maximal operator $S_\Lambda^*$. By Theorem 2 of \cite{Hughes}, in the range  $\frac{d}{d-2}<p<2$, this operator satisfies the same range of $\ell^p$-improving estimates as the single average for the range (as well as endpoint bounds at $p = \frac{d}{d-2}$). That is, we have
\begin{equation}\label{eq5.2}
\| S_\Lambda^* f \|_{p'} \lesssim \Lambda^{\frac d2 - \frac dp}\| f \|_p. 
\end{equation}

\begin{proof}[Proof of Corollary \ref{second refinement corollary}]
Suppose first that $p < 2$. By the triangle and H\"older's inequalities, 
\[
    \|T_\lambda(f,g)\|_{1} \lesssim \lambda^{1-\frac d2}\|A_\lambda (f)\|_{{p'}}\|g\|_{p}+{\lambda^{1-d}}\sum_{k=0}^{K}(2^k)^{\frac d2-1}(\lambda-2^k)^{\frac d2}\|A_{\lambda-2^k}(f)\|_{p}\|{S}_{2^k}^*(g)\|_{{p'}},
\]
where $K = \lfloor \log_2\lambda \rfloor$. Next, we appeal to the improving inequalities for both the spherical and ball averages: \eqref{eq3.2} and the inequality (a variant of \eqref{eq3.2} for $p=1$)
\[ \| A_\lambda(f) \|_q \lesssim \lambda^{-\frac d2 + \frac d{2q}} \|f\|_1. \]
This yields the upper bound
\begin{align*}
&\bigg( \lambda^{1-\frac d2}\lambda^{-\frac d2 +\frac{d}{2p'}} + \lambda^{1-d}\sum_{k=0}^{K}(2^k)^{\frac d2-1}(\lambda-2^k)^{\frac d2}(\lambda-2^k)^{-\frac{d}{2}+\frac{d}{2p}}(2^k)^{\frac{d}{2}-\frac{d}{p}} \bigg) \|f\|_{1}\|g\|_{p}\\
&\lesssim \bigg( \lambda^{1-\frac d2-\frac d{2p}} + \lambda^{1-d}\sum_{k=0}^{K}(2^k)^{d-1-\frac dp}(\lambda-2^k)^{\frac{d}{2p}} \bigg) \|f\|_{1}\|g\|_{p} \\
&\lesssim \bigg( \lambda^{1-\frac d2-\frac d{2p}} + \lambda^{1-d + \frac{d}{2p}}\sum_{k=0}^{K}(2^k)^{d-1-\frac dp} \bigg) \|f\|_{1}\|g\|_{p} \lesssim \lambda^{-\frac d{2p}}      \|f\|_{1}\|g\|_{p}.
 \end{align*}
Thus we get an overall bound of
\begin{equation*}
 \|T_\lambda(f,g)\|_{1} \lesssim \lambda^{-\frac{d}{2p}}\|f\|_{1}\|g\|_{p}.
\end{equation*}

When $p \ge 2$, the proof is more conventional, since the spherical maximal function is bounded on $\ell^p$. We switch the roles of the ball and spherical averages in the application of H\"older's inequality to get
\begin{align*}
\|T_\lambda(f,g)\|_{1} &\lesssim \lambda^{1-\frac d2}\|A_\lambda (f)\|_{{p'}}\|g\|_{p}+{\lambda^{1-d}}\sum_{k=0}^{K}(2^k)^{\frac d2-1}(\lambda-2^k)^{\frac d2}\|A_{\lambda-2^k}(f)\|_{p'}\|{S}_{2^k}^*(g)\|_{{p}} \\
&\lesssim \bigg( \lambda^{1-\frac d2}\lambda^{-\frac d2 +\frac{d}{2p'}} + \lambda^{1-d}\sum_{k=0}^{K}(2^k)^{\frac d2-1}(\lambda-2^k)^{\frac d2}(\lambda-2^k)^{-\frac{d}{2}+\frac{d}{2p'}} \bigg) \|f\|_{1}\|g\|_{p} \\
&\lesssim \bigg( \lambda^{1-\frac d2-\frac{d}{2p}} + \lambda^{1-d+\frac{d}{2p'}}\sum_{k=0}^{K}(2^k)^{\frac d2-1} \bigg) \|f\|_{1}\|g\|_{p} \lesssim \lambda^{-\frac d{2p}} \|f\|_{1}\|g\|_{p},
\end{align*}
and we arrive at the same conclusion as in the case $p<2$.
\end{proof}

When both $f$ and $g$ are summable, we may use symmetry to switch their roles. We may then interpolate between the bounds at $(r,1)$ and $ (1,r)$, with $\frac d{d-2} < r < 2$ to establish the following (more symmetric) version of Corollary \ref{second refinement corollary}.

\begin{cor}
\label{second refinement corollary symmetric}
If $f,g \in \ell^1(\mathbb Z^d)$, the bilinear spherical averages satisfy the following $\ell^p$-improving bound
\[
\|T_\lambda(f,g)\|_{1} \lesssim \lambda^{-\frac d2(\frac{1}{p}+\frac{1}{q}-1)}\|f\|_{p}\|g\|_{q}
\]
whenever $d \geq 5$, $p,q\geq 1$ and $1 < \frac{1}{p} + \frac{1}{q} < \frac{2d-2}{d}$.
\end{cor}

\begin{proof}
For a given pair $(p,q)$ satisfying the hypotheses, choose $r > \frac d{d-2}$ and $\theta \in (0,1)$ so that
\[ \frac 1r = \frac 1p + \frac 1q - 1, \quad \frac 1p = \theta + \frac {1-\theta}r, \quad \frac 1q = \frac {\theta}r + (1-\theta). \] 
These choices ensure that $\alpha = \theta p$ and $\beta = (1-\theta)q$ satisfy $0 < \alpha, \beta < 1$. We now set $s = \theta^{-1}$ and apply H\"older's inequality to get
\begin{align*}
\|T_\lambda(f,g)\|_1 &\leq \|T_\lambda(f^\alpha, g^{1-\beta})\|_s  \|T_\lambda(f^{1-\alpha}, g^{\beta})\|_{s'} \\
&= \big\| T_\lambda(f^{\alpha s}, g^{(1-\beta)s}) \big\|_1^{1/s}  \big\| T_\lambda(f^{(1-\alpha)s'}, g^{\beta s'}) \big\|_1^{1/s'}.    
\end{align*}
Note that $\alpha s = (1-\alpha)s'r = p$ and $\beta s' = (1-\beta)sr = q$. Applying the two $\ell^r$-improving bounds, we deduce that
\begin{align*}
\|T_\lambda(f,g)\|_1 &\lesssim \left( \lambda^{-\frac d{2r}} \| f^p \|_1 \big\| g^{(1-\beta)s} \big\|_r \right)^{1/s} \left( \lambda^{-\frac d{2r}} \big\| f^{(1-\alpha)s'} \big\|_r\| g^q \|_1  \right)^{1/s'} \\
&= \lambda^{-\frac d{2r}} \left(  \| f^p \|_1 \| g^q \|_1^{1/r} \right)^{1/s} \left( \| f^{p} \|_1^{1/r} \| g^q \|_1  \right)^{1/s'} = \lambda^{-\frac d{2r}} \| f \|_p \| g \|_q. \qedhere
\end{align*}
\end{proof}
{We note that Corollary \ref{second refinement corollary} provides a stronger bound than Proposition \ref{prop7} over the entire range $\frac {d}{d-2} < p < 2$. Indeed, even at the low endpoint of $p = \frac{d+1}{d-1}$, the exponent of $\lambda$ in Proposition \ref{prop7} approaches $\frac{3d -d^2}{2(d+1)}$, which exceeds the limiting value $1 - \frac d2$ of the exponent $\frac{-d}{2p}$ in the corollary as $p$ approaches $\frac{d}{d-2}$. Thus, at least when $g \in \ell^q(\mathbb Z^d)$ for all $q > \frac {d+1}{d-1}$, it is natural to interpolate Proposition \ref{prop7} and Corollary \ref{second refinement corollary} between these two bounds. Interpolation between Proposition \ref{prop7} at $p_1 = \frac {d+1}{d-1} + \delta$ and Corollary \ref{second refinement corollary} at $p_2 = \frac {d}{d-2} + \delta$, with $\delta > 0$ sufficiently small, gives}
\[ \|T_\lambda(f,g)\|_{1} \lesssim \lambda^{-\frac{d}{2p} + \frac {\theta d}{2p_1'}} \|f\|_{1}\|g\|_{p}, \]
where $\theta \in (0,1)$ is chosen so that $p^{-1} = \theta p_1^{-1} + (1-\theta)p_2^{-1}$. In particular, for any fixed $\eps > 0$, we may choose $\delta > 0$ so that
\[ \frac {\theta d}{2p_1'} - \frac d{2p} < \frac {d(d-(d-2)p)}{2p} - \frac d{2p} + \eps = \frac {d(d-1)}{2p} - \frac{d(d-2)}{2} + \eps. \]
This gives the following result.

\begin{cor}
When $d \ge 5$ and $\frac {d+1}{d-1} < p \le \frac {d}{d-2}$, the bilinear spherical averages satisfy the following $\ell^p$-improving bound
\[
\|T_\lambda(f,g)\|_{1} \lesssim_\eps \lambda^{\frac {d(d-1)}{2p} - \frac{d(d-2)}{2} + \eps}\|f\|_{1}\|g\|_{p}
\]
for any fixed $\eps > 0$.
\end{cor}

\section{Applications of $\ell^p$-improving estimates}

We now show how $\ell^p$-improving estimates can be leveraged into bounds dipping below the Banach setting for the discrete spherical averaging operator. We build on the setup of Iosevich, the third author and Sovine \cite{IPS22}, who in turn built on the work of Grafakos and Kalton \cite{GK01}. Due to the different form of $\ell^p$-improving estimates in the discrete setting, where the size of the radius matters, new features arise that are not present in the continuous setting. The following lemma plays a crucial role in our arguments.
  
\begin{lemma}\label{lp improving lemma}
Assume we have an $\ell^p$ improving estimate
\[
\| T_{\lambda}(f,g) \|_{1} \lesssim \lambda^{-\gamma}\|f\|_{p}\|g\|_{q}
\]
for $p,q$ such that $\frac{1}{r} := \frac{1}{p}+\frac{1}{q} > 1$. Then for each $s\in[r,1]$ we get
\[
\| T_{\lambda}(f,g) \|_{s} \lesssim \lambda^{\frac{d(1-s)}{2s}-\gamma} \| f \|_{p}\| g \|_{q}
\]
where the constants may depend on $s$ and $d$ but do not depend on $
\lambda$.
\end{lemma}

We note that this is specifically stated for the discrete bilinear spherical averages, but a general theorem of this type can easily be stated in the spirit of \cite{IPS22}. 

\begin{proof}[Proof of Lemma \ref{lp improving lemma}]
Tile $\mathbb{R}^d$ with a family of cubes $Q_{\bm{l}}$, $\bm{l}\in\mathbb{Z}^d$, each of sidelength slightly larger than $2\sqrt{\lambda}$. It is clear then that $T_{\lambda}(f\bm 1_{Q_{\bm{l}}},g\bm 1_{Q_{\bm{m}}}) = 0$ if $| \bm{l} - \bm{m} |_{\infty} > 1$ (where $| \cdot |_\infty$ is the usual $\infty$-norm on $\mathbb R^d$). Further, we will use that if $| \bm{l} - \bm{m} |_{{\infty}} \leq 1$ then the support of $T_{\lambda}(f\bm 1_{Q_{\bm{l}}},g\bm 1_{Q_{\bm{m}}})$ is contained in a larger cube of sidelength $3\sqrt{\lambda}$. Set $\mathcal D=\lbrace \bm{l}\in\mathbb{Z}^d : | \bm l |_{{\infty}} \leq 1 \rbrace$. Then for each $s\in[r,1)$, we get
\begin{align*}
\| T_{\lambda}(f,g) \|_{s}^s  &=  \sum_{\bm{x}\in\mathbb{Z}^d} \left| \sum_{\bm{l}\in\mathbb{Z}^d}\sum_{\bm{h}\in \mathcal D} T_{\lambda}(f\bm 1_{Q_{\bm{l}}},g\bm 1_{Q_{\bm{l}+\bm{h}}})(\bm{x}) \right|^{s}\\
&\lesssim \sum_{\bm{h}\in \mathcal D}  \sum_{\bm{x}\in\mathbb{Z}^d} \left| \sum_{\bm{l}\in\mathbb{Z}^d} T_{\lambda}(f\bm 1_{Q_{\bm{l}}},g\bm 1_{Q_{\bm{l}+\bm{h}}})(\bm{x}) \right|^{s} \\
&\lesssim \sum_{\bm{h}\in \mathcal D}  \sum_{\bm{x}\in\mathbb{Z}^d} \sum_{\bm{l} \in \mathbb{Z}^d} |T_{\lambda}(f\bm 1_{Q_{\bm{l}}},g\bm 1_{Q_{\bm{l}+\bm{h}}})(\bm{x}) |^{s}. 
\end{align*}
Here, we first use cancellation on far away cubes, then the quasi-triangle inequality, and finally Jensen's inequality for concave functions. We now use H\"{o}lder's inequality to deduce that
\begin{align*}
\| T_{\lambda}(f,g) \|_{s}^s  &\lesssim \sum_{\bm{h}\in \mathcal D} \sum_{\bm{l} \in \mathbb{Z}^d} \big\| T_{\lambda}(f\bm 1_{Q_{\bm{l}}},g\bm 1_{Q_{\bm{l}+\bm{h}}}) \big\|_1^s \cdot \big| \mathrm{supp} \, T_{\lambda} (f\bm 1_{Q_{\bm{l}}},g\bm 1_{Q_{\bm{l}+\bm{h}}}) \big|^{1-s} \\
&\lesssim \lambda^{\frac{d(1-s)}{2}} \sum_{\bm{h}\in \mathcal D} \sum_{\bm{l} \in \mathbb{Z}^d} \big\| T_{\lambda}(f\bm 1_{Q_{\bm{l}}},g\bm 1_{Q_{\bm{l}+\bm{h}}}) \big\|_1^s.
\end{align*}
We finally use our assumed $\ell^p$-improving estimate, the nestedness of $\ell^p$-spaces, and H\"{o}lder's inequality (noting that $\frac rp + \frac rq = 1$) to conclude that
\begin{align*}
\| T_{\lambda}(f,g) \|_{s}  &\lesssim \lambda^{\frac{d(1-s)}{2s}} \left( \sum_{\bm{h}\in \mathcal D} \sum_{\bm{l} \in \mathbb{Z}^d} \lambda^{-s\gamma} \| f\bm 1_{Q_{\bm{l}}} \|_{p}^{s} \| g\bm 1_{Q_{\bm{l}+\bm{h}}} \|_{q}^{s} \right)^{1/s} \\
&\lesssim \lambda^{\frac{d(1-s)}{2s} - \gamma} \left( \sum_{\bm{h}\in \mathcal D} \sum_{\bm{l} \in \mathbb{Z}^d} \| f\bm 1_{Q_{\bm{l}}} \|_{p}^{r} \| g\bm 1_{Q_{\bm{l}+\bm{h}}} \|_{q}^{r} \right)^{1/r} \\
&\lesssim \lambda^{\frac{d(1-s)}{2s}-\gamma} \sum_{\bm{h}\in \mathcal D} \left( \sum_{\bm{l}\in\mathbb{Z}^d} \| f\bm 1_{Q_{\bm{l}}} \|_{p}^{p} \right)^{\frac{1}{p}} \left( \sum_{\bm{l}\in\mathbb{Z}^d} \| g\bm 1_{Q_{\bm{l + h}}} \|_{q}^{q} \right)^{\frac{1}{q}}\\
&\lesssim \lambda^{\frac{d(1-s)}{2s}-\gamma} \| f \|_{p}\| g \|_{q}. \qedhere
\end{align*}
\end{proof}

\begin{remark}
    We make some quick observations to put our ranges obtained via Lemma \ref{lp improving lemma} into perspective.  Restricting to the range $\frac{1}{r} \leq \frac{1}{p}+\frac{1}{q}$ (see \cite{AP} for a more thorough discussion of why this range is necessary for the maximal variant), the method in Corollary \ref{second refinement corollary} gives $r > \frac{d}{2d-2}$ from calculating $r$ with $p = \frac{d}{d-2}$ and $q = 1$.  Due to the nesting property, a bound for $(p_0,q_0)$ carries over to a bound for $(p,q)$ for all $p \geq p_0$, $q \geq q_0$ (that is, $\frac{1}{p}+\frac{1}{q} \leq \frac{1}{p_0}+\frac{1}{q_0}$), which yields $1/r\leq 2$, a much more simplified range than in the continuous setting.  This simplification stems from the fact that there are no balls with arbitrarily small radius in the discrete setting (the characteristic function of a ball with small radius dictates the upper bound on $\frac{1}{p}+\frac{1}{q}$ in \cite{JL}, and that counterexample does not apply in our situation).  It is worth noting that if the range of $\ell^p$-improving is extended from what is currently known, our bounds will subsequently improve.

Necessary conditions in the linear (discrete) setting are listed in \cite{Hughes} and \cite{KL}, however, two different ranges are given.  Details in the proof of the necessary range are quite sparse, and the example given there relates to deep, difficult problems in number theory, and connects to counting simplices.  Due to this background and the even greater complexity that the bilinear situation entails, we plan to look more deeply into all of these necessary conditions in future work, while pursuing sharper bounds for the simplex operators via number theory.
\end{remark}

Finally, in light of the above lemma, we revisit the implications for our results from \S\ref{sec:4}. Recall Proposition \ref{prop7}. We leverage that result through an application of the lemma.

\begin{proof}[Proof of Corollary \ref{first refinement corollary}]
We apply Lemma \ref{lp improving lemma} with $p=q$ and $\gamma = -\frac d2 + \frac dp$, so $\frac{1}{r} = \frac{2}{p}$.  We get
\begin{equation*}
\| T_{\lambda}(f,g) \|_{s} \lesssim \lambda^{\frac{d(1-s)}{2s} + \frac{d}{2}-\frac{d}{p}} \| f \|_{p}\| g \|_{p} \lesssim \lambda^{\frac{d}{2s} -\frac{d}{p}} \| f \|_{p}\| g \|_{p},
\end{equation*}
which is $\ell^p$-improving on the full range $\frac p2 \le s\leq 1$ for $\frac{d+1}{d-1} < p < 2$.
\end{proof}

\subsection*{Acknowledgments}
The first author was supported by NSF grant DMS-1954407. The third author was supported by Simons Foundation Grant \#360560. 

\bibliographystyle{amsplain}

\end{document}